\documentclass[12pt,a4paper,twoside]{article}
\usepackage{booktabs}
\usepackage{multirow}
\usepackage{geometry}
\usepackage[UKenglish]{babel}
\usepackage[T1]{fontenc}
\usepackage[utf8x]{inputenc}
\usepackage{latexsym,amsfonts,amsmath,amsthm,amssymb,mathrsfs}
\usepackage{fullpage}
\usepackage{xcolor,bm, bbm}
\usepackage{paralist}
\usepackage{todonotes}
\usepackage{dsfont}
\usepackage{float}
\usepackage{caption}
\usepackage{xfrac}
\usepackage[colorlinks, linkcolor=blue,citecolor=blue, anchorcolor=blue,urlcolor=blue,hypertexnames=false]{hyperref}

\usepackage{graphicx}
\usepackage{tikz,ifthen}
\usepackage{pgfplots}
\newtheorem{thm}{Theorem}%[section]
\newtheorem{lem}{Lemma}[section]

\newtheorem{cor}[thm]{Corollary}

\newtheorem{proposition}[lem]{Proposition}

\newtheorem*{rem*}{Remark}
\newtheorem*{thm*}{Theorem}

\title{ Fourier Dimension in Inhomogeneous Duffin--Schaeffer Conjecture}
\author{
  Bo Tan\\
  {\normalsize School of Mathematics and Statistics}\\
  {\normalsize Huazhong University of Science and Technology, 430074 Wuhan, PR China}\\
  {\normalsize Email: \texttt{tanbo@hust.edu.cn}}
  \and
  Qing-Long Zhou\footnote{Corresponding author.}\\
  {\normalsize School of Mathematics and Statistics}\\
  {\normalsize Wuhan University of Technology, 430070 Wuhan, PR China}\\
  {\normalsize Email: \texttt{zhouql@whut.edu.cn}}
}

\date{}

\begin{document}

%\subjclass{11K60 11J83 11J71 11H16}
%keywords: metric Diophantine approximation, Duffin-Schaeffer conjecture, Fourier dimension

\maketitle

\begin{abstract}
Let \(Q \subseteq \mathbb{N}\) be a subset, and let \(\psi\colon \mathbb{N} \to [0, \tfrac{1}{2})\), \(\theta\colon \mathbb{N} \to \mathbb{R}\) be functions. Let \(\{A_q\}\) and \(\{B_q\}\) be sequences of integers such that \(\gcd(A_q, B_q) = 1\) and \(B_q > 0\) for all \(q\). Define \(W_Q^{\ast}(\psi,\theta)\) to be the set of \(x \in [0,1]\) for which
\[
\left| x - \frac{p + \theta(q)}{q} \right| < \frac{\psi(q)}{q}
\]
holds for infinitely many \((p,q) \in \mathbb{Z} \times Q\) with \(\gcd(B_q p + A_q, q) = 1\).

In this paper, we determine the Fourier dimension of \(W_Q^{\ast}(\psi,\theta)\). Our result not only recovers the classical theorems of Kaufman and Bluhm (concerning the homogeneous case \(\psi(q) = q^{-\tau}\) with \(\tau \ge 1\)) and the one-dimensional version of a result by Cai and Hambrook on the inhomogeneous approximable set, but also provides a complete inhomogeneous generalization. Moreover, it gives an affirmative answer to the coprime formulation of the Chen--Xiong conjecture.
\end{abstract}

\section{Motivation$\colon$ Metrical Diophantine approximation}
 A central theme in Diophantine approximation is to understand how well real numbers can be approximated by rationals. This topic lies at the heart of metric number theory, where one often studies the size---measured in terms of the Lebesgue measure, Hausdorff measure/dimension, Fourier dimension, etc.---of Diophantine sets satisfying prescribed approximation conditions. A fundamental problem in this context is, given an approximating function \(\psi\colon \mathbb{N} \to [0, 1/2)\), to determine the size of the \(\psi\)-well approximable set for systems of linear forms:
\[
W(n,m;\psi,\bm{\theta}) := \left\{ \mathbf{x} \in [0,1]^{nm} \;:\; |\mathbf{q} \mathbf{x} - \mathbf{p} - \bm{\theta}| < \psi(|\mathbf{q}|) \text{ for i.m. } (\mathbf{p}, \mathbf{q}) \in \mathbb{Z}^m \times \mathbb{N}^n \right\},
\]
where ``infinitely many'' is abbreviated as ``i.m.''. Here, \(\mathbf{q} \in \mathbb{N}^n\), \(\mathbf{p} \in \mathbb{Z}^m\), and \(\bm{\theta} \in \mathbb{R}^m\) are regarded as row vectors, while \(\mathbf{x} \in [0,1]^{nm}\) is interpreted as an \(n \times m\) matrix. Throughout this paper, \(|\cdot|\) denotes the maximum norm$\colon$ \(|\mathbf{q}| = \max_{1 \le j \le n} |q_j|\). To set the stage, we briefly review the relevant metric results concerning \(W(n,m;\psi,\bm{\theta})\), which will help motivate the present work.

 \medskip

   \noindent{\bf Lebesgue Measure.}
{\em{%Inhomogeneous
Khintchine--Groshev Theorem}} \cite{G38,S64} established  that for any function $\psi\colon \mathbb{N}\to [0,\frac{1}{2})$ that is monotonically decreasing, the Lebesgue measure of the set $W(n,m;\psi,\bm{\theta})$ satisfies
\begin{equation*}
\mathcal{L}\big(W(n,m;\psi,\bm{\theta})\big)=\begin{cases}
   0   & \text{if $\sum_{q=1}^{\infty}q^{n-1}\psi(q)^{m}<\infty$}, \\
   ~&\\
    1  & \text{if $\sum_{q=1}^{\infty}q^{n-1}\psi(q)^{m}=\infty$}.
\end{cases}
\end{equation*}
The non-monotonic case is more involved.
In the homogeneous setting $(\bm{\theta}=\bm{0})$, Beresnevich and Velani \cite{BV10} removed the monotonicity condition when $nm\ge 2.$ However, monotonicity is necessary in the case $(m,n) = (1,1),$
as demonstrated by the counterexample of Duffin and Schaeffer \cite{DS41}.
Further,  consider the set
$$ {W}^{\ast}(1,1;\psi,0):=\Big\{x\in [0,1]\colon  |qx-p|<\psi(q)\text{ for i.m.} (p,q)\in \mathbb{Z}\times \mathbb{N} \text{ with } \gcd(p,q)=1\Big\}.$$
{\em{Duffin--Schaeffer conjecture}} asserted that
\begin{equation*}
\mathcal{L}\big(W^{\ast}(1,1;\psi,0)\big)=\begin{cases}
   0   & \text{if $\sum_{q=1}^{\infty}\frac{\phi(q)}{q}\psi(q)<\infty$}, \\
   ~&\\
    1  & \text{if $\sum_{q=1}^{\infty}\frac{\phi(q)}{q}\psi(q)=\infty$},
\end{cases}
\end{equation*}
where $\phi$ is Euler's totient function.
 This conjecture animated a great deal of research until it was finally proved in a breakthrough by Koukoulopoulos and Maynard \cite{KM20}.

In the inhomogeneous case \((\bm{\theta} \neq \bm{0})\), Allen and Ram\'{i}rez \cite{AR23} removed the monotonicity condition for cases with \(nm \ge 3\). For \((n,m) = (2,1)\), Hauke \cite{H25} showed that the monotonicity condition can be removed for any non-Liouville irrational inhomogeneous parameter \(\theta\). Kim \cite{K25} further proved that this result extends to all \(\theta \in \mathbb{R}\) (without the non-Liouville restriction) in the case \((n,m) = (2,1)\), and also holds for \((n,m) = (1,2)\) with any rational parameter \(\bm{\theta} \in \mathbb{Q}^2\). The other cases for \((n,m) = (1,2)\) are still wide open. Similarly to the homogeneous
case for $(n,m)=(1,1)$, counterexamples to the removal of monotonicity were given in \cite{CHPR25,R17}. Despite rumours of its falsity, the most direct analogue of the Duffin--Schaeffer conjecture remains formally open (see \cite{R17}):  set
$$ {W}^{\ast}(1,1;\psi,\theta):=\Big\{x\in [0,1]\colon  |qx-p-\theta|<\psi(q)\text{ for i.m. } (p,q)\in \mathbb{Z}\times \mathbb{N} \text{ with } \gcd(p,q)=1\Big\}.$$
{\em{Inhomogeneous Duffin--Schaeffer conjecture}} states that
\begin{equation*}
\mathcal{L}\big(W^{\ast}(1,1;\psi,\theta)\big)=\begin{cases}
   0   & \text{if $\sum_{q=1}^{\infty}\frac{\phi(q)}{q}\psi(q)<\infty$}, \\
   ~&\\
    1  & \text{if $\sum_{q=1}^{\infty}\frac{\phi(q)}{q}\psi(q)=\infty$}.
\end{cases}
\end{equation*}
The so-called weak Duffin--Schaeffer conjecture (where the coprimality condition is dropped, so that the set \(W\) is used in place of \(W^{\ast}\)) also remains open; see \cite[Conjecture 1.22]{CT24}. This weak version was resolved for a wide class of non-monotonic functions in \cite{C18,CT19,CT24}. Results under extra divergence conditions were established by Yu \cite{Y19,Y21}. For historical details, we refer to Tables 1 and 2 in \cite{HR24}.

 \medskip

   \noindent{\bf Hausdorff Dimension.}
Jarn\'{i}k's theorem \cite{J31} establishes that, under the monotonicity assumption on $\psi$, the Hausdorff dimension of $W(1,1;\psi,0)$ is
\[
\dim_{\mathrm{H}} W(1,1;\psi,0) = \min\{t(\psi),1\},
\]
where
\[
t(\psi) = \inf\left\{s \ge 0 : \sum_{q \in \mathbb{N}} q \left( \frac{\psi(q)}{q} \right)^s < \infty \right\}.
\]
%Here, $\dim_{\mathrm{H}} W(1,1;\psi,0)$ is the Hausdorff dimension of $W(1,1;\psi,0).$
(In this paper, we do not need the definition of Hausdorff dimension; the interested reader may find further details in \cite{F14}.) An alternative proof of this result can be derived from Khintchine's theorem by applying the mass transference principle due to Beresnevich and Velani \cite{BV06}. For a general (not necessarily monotonic) function $\psi$, the Hausdorff dimension of the set $W(1,1;\psi,0)$ was thoroughly investigated by Hinokuma and Shiga \cite{HS96}. More recently, Yu \cite{Y19} showed that
\[
\dim_{\mathrm{H}} W^{\ast}(1,1;\psi,\theta) = \dim_{\mathrm{H}} W(1,1;\psi,0).
\]
A straightforward covering argument yields $\dim_{\mathrm{H}} W(1,1;\psi,\theta) \le \min\{t(\psi),1\}$. Consequently, we obtain the following identities:
\[
\dim_{\mathrm{H}} W(1,1;\psi,\theta) = \dim_{\mathrm{H}} W^{\ast}(1,1;\psi,\theta) = \dim_{\mathrm{H}} W(1,1;\psi,0).
\]
For details and a comprehensive overview of the Hausdorff dimension of $ W(n,m;\psi,{\bm \theta})$ in higher dimensions, see  \cite{WW21,BV26} and the references therein.

 \medskip

   \noindent{\bf Fourier Dimension.}  %We now turn our attention to the Fourier dimension of sets $W(n,m;\psi,{\bm \theta})$.
Before proceeding, we recall some definitions.
%Throughout, for a real vector $\mathbf{x} = (x_1, \ldots, x_d)$, the notation %$|\mathbf{x}|_\infty$ refers to the maximum norm, defined as
%\[
%|(x_1, \ldots, x_d)|_\infty = \max_{1 \leq i \leq d} |x_i|.
%\]
As usual, the Fourier transform of a non-atomic probability measure $\nu$ is defined by
\[
\hat{\nu}({\bm t}) := \int e^{-2\pi i \langle {\bm t}, \mathbf{x} \rangle} \, {\mathrm{d}}\nu(\mathbf{x}) \qquad ({\bm t} \in \mathbb{R}^d),
\]
where \(
\langle {\bm t}, \mathbf{x} \rangle = t_1 x_1 + \cdots + t_d x_d.
\)
Loosely speaking, Fourier decay reflects the ``smoothness" or ``randomness"  of a measure: it indicates that the measure is spread out in such a way that coherent long-range correlations are dampened. In harmonic analysis and geometric measure theory, this decay is intimately connected to the fractal and dimensional properties of the measure’s support.  Indeed, a classical result of Frostman states that for $E\subset \mathbb{R}^d$,
\begin{equation}\label{Frostman}
\dim_{\rm H} E \geq \dim_{\rm F} E,
\end{equation}
where $\dim_{\rm F}E$ is the Fourier dimension of $E$, defined as
\[
\dim_{\rm F} E = \sup \left\{ s \in [0, d] : \exists \nu \in \mathcal{P}(E) \text{ such that } |\hat{\nu}({\bm \xi})| \ll_s (1 + |{\bm \xi}|)^{-s/2}, {\bm \xi}\in\mathbb{R}^{d} \right\}.\footnote{Throughout we use Vinogradov notation$\colon$ $A\ll B$ means $|A|\le C |B|$ for some constant $C>0$ independent of $A$ and $B$; $A\asymp B$ means $A\ll B$ and $B\ll A.$}
\]
Here \(\mathcal{P}(E)\) denotes the set of Borel probability measures on \(\mathbb{R}^d\) that give full measure to \(E\). When we have equality in (\ref{Frostman}), the set $E$ is called a Salem set.

A rich variety of random Salem sets is known to exist; we refer the interested reader to \cite{M15} for further details. Non-trivial deterministic fractal Salem or non-Salem sets are, however, hard to construct. This serves as one of the main motivations for the present paper. We survey the known results on Salem and non-Salem sets arising in Diophantine approximation, providing both a historical overview and a discussion of recent progress.

\begin{table}[H]
\centering
\renewcommand{\arraystretch}{1.3}
\caption{Explicit constructions of Salem sets in Euclidean and $p$-adic spaces}
\label{tab:salem_construction}
\begin{tabular*}{\textwidth}{@{\extracolsep{\fill}} l l l @{}}
\toprule
Space & Setting type & Reference \\
\midrule
\multirow{3}{*}{$\mathbb{R}$}
 & $\tau$-well approximable  & Kaufman \cite{K81}, Bluhm \cite{B98} \\
\addlinespace[3pt]
 & restricted $\tau$-well approximable   & Hambrook \cite{H19} \\
 \addlinespace[3pt]
 & $\tau$-Dirichlet non-improvable   & Tan--Zhou \cite{TZ26} \\
 \midrule
\addlinespace[3pt]
$\mathbb{R}^d$  & \shortstack[c]{``algebraic'' variants of \\ $\tau$-well approximable}   & \shortstack[c]{Hambrook \cite{H17} ($d=2$) \\ Fraser--Hambrook \cite{FH23}} ($d\ge2$)\\
\midrule
\addlinespace[3pt]
$\mathbb{Q}_p$  & $\tau$-well approximable  & Fraser--Hambrook \cite{FH20}\\
\bottomrule
\end{tabular*}
\end{table}

\begin{table}[H]
\centering
\renewcommand{\arraystretch}{1.3}
\caption{Explicit constructions of non-Salem sets in Euclidean space}
\label{tab:salem_progress}
\begin{tabular*}{\textwidth}{@{\extracolsep{\fill}} l l l @{}}
\toprule
Space & Setting type & Reference \\
\midrule
$\mathbb{R}$ &  $\psi$-well approximable & Chen--Xiong \cite{CX26} \\
\midrule
$\mathbb{R}$ &  uniformly $\tau$-well approximable & Li--Liu \cite{LL26} \\
\midrule
$\mathbb{R}^d$  &  multiplicative $\tau$-well approximable & Tan--Zhou \cite{TZ25} ($d=2$), He \cite{H26} ($d\ge 2$)\\
\midrule
\multirow{1}{*}{$\mathbb{R}^d$}
& \shortstack[c]{$\tau$-well approximable vectors \\ badly approximable vectors} & Hambrook--Yu \cite{HY23} \\
\midrule
\addlinespace[3pt]
$\mathbb{R}^d$  & well-approximable matrices &  Cai--Hambrook \cite{CH24} \\
\bottomrule
\end{tabular*}
\end{table}

\noindent\textbf{Remark 1.} Let us make the following remarks regarding Tables 1 and 2:
\begin{itemize}
\item As summarized in Table 1, the \(\tau\)-well approximable set
\[
\left\{ x \in [0,1] : |qx - p| \le q^{-\tau} \text{ for i.m. } (p,q) \in \mathbb{Z} \times \mathbb{N} \right\}
\]
is Salem. However, the proofs by Kaufman and Bluhm of the lower bounds for the Fourier dimension of such sets essentially rely on considering the coprime setting, i.e., restricting to rational approximations \(p/q\) with \(\gcd(p,q) = 1\):
\[
\left\{ x \in [0,1] : |qx - p| \le q^{-\tau} \text{ for i.m. } (p,q) \in \mathbb{Z} \times \mathbb{N} \text{ with } \gcd(p,q) = 1 \right\}.
\]
To the best of our knowledge, the Fourier dimension in the coprime setting remains unknown for general approximation functions. Our goal is to establish a complete characterization in the general setting.

\item  Let $Q \subseteq \mathbb{N}$ be  a subset, and define the associated approximable set
\[
W_Q(1,1;\psi,\theta) = \left\{ x \in [0,1] : |qx - p - \theta| < \psi(q) \text{ for i.m. } (p,q) \in \mathbb{Z} \times Q \right\}.
\]
Specially, $W_{\mathbb{N}}(1,1;\psi,\theta)=W(1,1;\psi,\theta).$ Cai and Hambrook \cite[Theorem 1.4.1]{CH24}{\footnote{Cai and Hambrook studied the mutli-dimensional generalization of $W_Q(1,1;\psi,\theta)$}} proved that under the convergence condition $\sum_{q \in Q} \psi(q) < \infty$,
\[
\dim_{\mathrm{F}} W_Q(1,1;\psi,\theta) = \min\{ 2 s(Q,\psi), 1 \},
\]
where
\[
s(Q,\psi) = \inf \left\{ s \ge 0 : \sum_{q \in Q} \left( \frac{\psi(q)}{q} \right)^s < \infty \right\}.
\]
They further proposed the following conjecture: prove or disprove that
\[
\text{if } \sum_{q\in Q}\psi(q)=\infty,\text{ then } \dim_{{\rm F}}W_Q(1,1;\psi,\theta)=1.
\]
The authors \cite{TZ25} constructed a function \( \psi \) with
\(
\sum_{q \in \mathbb{N}} \psi(q) = \infty
\)
but
\[
\dim_{\rm F} W(1, 1; \psi, 0) = 0,
\]
providing a negative answer to this conjecture. Chen and Xiong recently established the complete result for the set $W(1,1;\psi,0)$: if
\begin{equation}\label{con2}
\sum_{q\in \mathbb{N}}\phi(q)\sup_{q\mid n}\frac{\psi(n)}{n}<\infty,
\end{equation}
then
\[
\dim_{\rm F}W(1,1;\psi,0)=\min\{2s(\psi),1\},
\]
where
\[
s(\psi)=\inf\left\{s\ge 0\colon \sum_{q\in \mathbb{N}}\left(\frac{\psi(q)}{q}\right)^{s}<\infty\right\}.
\]
In the inhomogeneous case, they also proved that the dimension formula remains valid when $\theta$ is rational, i.e., $\theta = A/B$ with $\gcd(A,B)=1$, provided that condition \eqref{con2} holds.
Furthermore, they proposed the following conjecture:
\[\text{the dimension formula remains valid for } \theta\in \mathbb{R}\setminus\mathbb{Q}. \]
While the conjecture remains unresolved, it motivates the investigation of related problems where additional arithmetic constraints are imposed, for instance, the coprime setting.

\item There also exist several natural and important open problems. For example$\colon$

(1) The badly approximable set
\[
\textbf{Bad} := \left\{ x \in [0,1] : \inf_{q \in \mathbb{N}} q \|qx\| > 0 \right\}
\]
has null Lebesgue measure but full Hausdorff dimension, where  \(\|\cdot\|\) denotes the distance from a real number to the nearest integer. Currently, it is only known that \(\textbf{Bad}\) is of positive Fourier dimension \cite{K80,QR03}; whether this set is a Salem set remains an open problem.

(2) The Fourier dimension of \(W(n,m;\psi,\bm{\theta})\) established by Cai and Hambrook only provides isotropic estimates of Fourier decay for square targets (i.e., when the approximating functions are identical in all directions). What can be said about the Fourier dimension when the target is a rectangle (especially a very elongated one), whose Fourier transform exhibits significant directional decay? Precisely, consider  the weighted \({\bf \Psi}\)-well approximable set
\[\left\{{\bf x}\in [0,1]^{nm}\colon |{\bf qx}-{\bf p}-\bm{\theta}|<\psi_j(|{\bf q}|)~ (1\le j\le m) \text{ for i.m. } ({\bf p},{\bf q})\in \mathbb{Z}^{m}\times \mathbb{N}^{n}\right\},\]
where  $\psi_j\colon\mathbb{N}\to [0,\frac{1}{2})$ for $1\le j\le m.$ He \cite{H26} proved that if
\[\sum_{\mathbf{q} \in \mathbb{N}^n} \prod_{j=1}^m \psi_j(|\mathbf{q}|) < \infty\]
   and  \(s(\{\psi_j\}) < 1\), where
\[
  s(\{\psi_j\}) := \inf \left\{ s \geq 0 : \sum_{\mathbf{q} \in \mathbb{N}^n } \left( \min_{1 \leq j \leq m} \frac{\psi_j(|\mathbf{q}|)}{|\mathbf{q}|} \right)^s < \infty \right\},
\]
then the set has Fourier dimension \(2s(\{\psi_j\})\). However, whether this result  holds in the general case remains unknown.

\item[] {\textbf{Applications.}} There exists a profound connection between positive Fourier dimension and the existence of normal numbers in fractal sets, as discussed in the seminal work of Pollington \emph{et al.} \cite{PVZZ22}. Furthermore, constructing measures with prescribed Fourier decay on Diophantine sets has far-reaching implications. It can facilitate the development of the theory of multiplicative and simultaneous Diophantine approximation on fractals \cite{PV00,CY24,TZ26+}, as well as twisted Diophantine approximation \cite{HR25,CZ25}. Consequently, Fourier dimension can be regarded as a critical bridge linking harmonic analysis and metric number theory.

\end{itemize}

Based on the above discussion and motivated by open problems in inhomogeneous Diophantine approximation with coprimality conditions, we   study a concrete arithmetic framework that naturally incorporates a condition of Duffin--Schaeffer type.

Let $\{A_q\}$ and $\{B_q\}$ be sequences of integers such that $\gcd(A_q, B_q)=1$ and $B_q > 0$ for all $q$. Let $Q \subseteq \mathbb{N}$ be a  subset, let $\psi\colon Q \to [0, \tfrac{1}{2})$ be a positive function, and let $\theta\colon Q \to \mathbb{R}$ be an arbitrary function.
Define $W_{Q}^{\ast}(\psi,\theta)$ to be the set of $x\in[0,1]$ such that
\[
\left| x - \frac{p+\theta(q)}{q} \right|<\frac{\psi(q)}{q} \text{ for i.m. }(p,q)\in \mathbb{Z}\times Q \text{ with } \gcd(B_qp+A_q,q)=1.
\]
Strictly speaking, the set \({W}_{Q}^{\ast}(\psi,\theta)\) depends on the sequences \(\{A_q,B_q\}\). Since these sequences are fixed throughout our study, we omit this dependence from the notation. Thanks to the limsup set structure of \(W_{Q}^*(\psi,\theta)\), it follows from \cite[Proposition 1]{BHV24} and the first Borel--Cantelli lemma that the Lebesgue measure of $W^*(\psi,\theta)$ vanishes provided that
\begin{equation}\label{con}
\sum_{q\in Q} \psi(q) \prod_{{p\mid q, ~p\nmid B_q}} \bigl(1 - p^{-1}\bigr) < \infty,
\end{equation}
where the product is over all prime numbers $p$ that divide $q$ and are coprime to $B_q$.
Our main result provides the Fourier dimension of this set.

\begin{thm}\label{thm1}
Let $Q\subset \mathbb{N}$, and let $\psi\colon Q\to [0,\frac{1}{2})$ be a function satisfying the condition (\ref{con}).
Then the Fourier dimension of \(W_{Q}^*(\psi,\theta)\) is
$$\dim_{{\rm F}}W_Q^{\ast}(\psi,\theta)=\min\{2s(Q,\psi),1\}.$$
\end{thm}
When \(Q = \mathbb{N}\), we abbreviate \(W_{\mathbb{N}}^{\ast}(\psi,\theta)\) as \(W^{\ast}(\psi,\theta)\). A direct corollary of Theorem \ref{thm1} obtained by taking \(Q = \mathbb{N}\) reads as follows.

\begin{cor}\label{cor2}
Let \(\psi \colon \mathbb{N} \to [0, 1/2)\) be a function satisfying condition \eqref{con} with \(Q\) replaced by \(\mathbb{N}\). Then
\[
\dim_{\mathrm{F}} W^{\ast}(\psi,\theta) = \min\{2s(\psi), 1\}.
\]
\end{cor}

\noindent\textbf{Remark 2.} Although Corollary~\ref{cor2} is a special case of Theorem~\ref{thm1}, they are essentially equivalent. Indeed, assuming Corollary~\ref{cor2} holds, we prove Theorem~\ref{thm1} as follows.

We extend the function $\psi: Q \to [0, 1/2)$ to $\overline{\psi}: \mathbb{N} \to [0, 1/2)$ by setting
\[
\overline{\psi}(q) =
\begin{cases}
\psi(q) & \text{if } q \in Q, \\[4pt]
0 & \text{if } q \notin Q.
\end{cases}
\]
(The value of $\theta$ on $\mathbb{N}\setminus Q$ may be defined arbitrarily.) One then readily verifies that
\[
W_Q^\ast(\psi, \theta) = W^\ast(\overline{\psi}, \theta)
\qquad\text{and}\qquad
s(Q,\psi) = s(\overline{\psi}).
\]
\medskip

\noindent\textbf{Remark 3.} Let us make more remarks regarding Theorem \ref{thm1}$\colon$

\begin{itemize}

\item
In Theorem \ref{thm1}, the Duffin--Schaeffer type condition \(\gcd(A_q + p B_q, q) = 1\) --- where \(A_q\) and \(B_q\) are coprime integers --- is a recurring theme in the literature, appearing in early works such as \cite{S64} and more recent contributions like \cite{BHV24, CT24, K25}. In fact, the formulation of this condition is, to some extent, inspired by Dirichlet's theorem: for any \(\theta \in \mathbb{R}\) and \(q \in \mathbb{N}\), there exists a pair \((A_q, B_q) \in \mathbb{Z} \times \mathbb{N}\) such that
\[
|B_q \theta - A_q| < \frac{1}{q},\qquad 1 \le B_q \le q,\qquad \gcd(A_q, B_q) = 1.
\]
Then \(\left| x -  ({p + \theta})/{q} \right|\) can be well approximated by \(\left| x - ({B_q p + A_q})/({B_q q}) \right|\). Therefore, analogous to the homogeneous Duffin--Schaeffer conjecture, the restriction
\[
\gcd(B_q p + A_q, B_q q) = 1, \text{ or equivalently, }     \gcd(B_q p + A_q, q) = 1
\]
is imposed.

\item
The result  \cite[Theorem 1.4.1]{CH24} of Cai and Hambrook  does not apply directly to our setting for two main reasons. First, our set \(W^*(\psi,\theta)\) involves an additional coprimality constraint \(\gcd(B_q p + A_q, q) = 1\) and is governed by the different convergence condition \eqref{con} (if \(Q = \mathbb{N}\), the condition \eqref{con} is weaker than the one required in \cite{CH24}). Second, a natural approach to obtaining a lower bound for \(\dim_{\mathrm{F}} W^*(\psi,\theta)\) would be to select a suitable subset \(Q \subseteq \mathbb{N}\) such that \(W_Q(1,1;\psi,\theta) \subseteq W^*(\psi,\theta)\). A concrete choice is to take \(Q\) as the set of all prime numbers. However, a direct comparison of the associated \(s\)-parameters reveals that \(s(\psi) \ge s(Q,\psi)\). Consequently, the lower bound provided by the Cai and Hambrook theorem is strictly weaker than the desired lower bound \(2s(\psi)\) for \(W^*(\psi,\theta)\). Hence, establishing the sharp Fourier dimension estimate in our setting calls for a nontrivial refinement of their method.

\item Putting \(A_q = 0\) and \(B_q = 1\) for all \(q \in \mathbb{N}\), we find that Theorem \ref{thm1} not only recovers the classical results of Kaufman and Bluhm (which deal with the homogeneous case and \(\psi(q) = q^{-\tau}\) with \(\tau \ge 1\)) but also establishes a general inhomogeneous version. At the same time, Theorem \ref{thm1} provides an affirmative answer to the coprime version of the Chen and Xiong conjecture.

\item  In the special case where \(A_q = 1\) and \(B_q = q\) for all \(q \in \mathbb{N}\), our result also recovers Theorem 1.4.1 of \cite{CH24} in the one-dimensional setting.

\item Our method may conceivably be applicable to the study of  Fourier dimension in the Duffin–Schaeffer setting for multiplicative and weighted Diophantine approximation, although we do not assert this with certainty.

\end{itemize}

\section{Proof of Theorem \ref{thm1}}

As claimed before, it suffices to prove Corollary \ref{cor2}.
To begin, we recall the limsup set structure of \(W^{\ast}(\psi,\theta)\). For $q\in\mathbb{N},$ define
\begin{equation}\label{I_q}
  \mathcal{I}_q=\Big\{p\in\{0,1,\ldots,q-1\}\colon \gcd(B_qp+A_q,q)=1\Big\}.
\end{equation}
Then
\begin{align*}
W^{\ast}(\psi,\theta)
&=\left\{x\in[0,1]\colon \left|x-\frac{p+\theta(q)}{q}\right|<\frac{\psi(q)}{q} \text{ for i.m. } (p,q)\in \mathbb{Z}\times \mathbb{N},~ \gcd(B_qp+A_q,q)=1\right\}\\
&=\bigcap_{Q=1}^{\infty}\bigcup_{q=Q}^{\infty}\left\{x\in[0,1]\colon \left|x-\frac{p+\theta(q)}{q}\right|<\frac{\psi(q)}{q} \text{ for some } p\in \mathbb{Z},~\gcd(B_qp+A_q,q)=1\right\}\\
&=\bigcap_{Q=1}^{\infty}\bigcup_{q=Q}^{\infty}\left\{x\in[0,1]\colon \left\|x-\frac{p+\theta(q)}{q}\right\|<\frac{\psi(q)}{q} \text{ for some } p\in \mathcal{I}_q\right\}\\
&=:\bigcap_{Q=1}^{\infty}\bigcup_{q=Q}^{\infty}A_q^{\mathcal{I}},    \end{align*}
where \(A_q^{\mathcal{I}}\) can be written explicitly as follows:
\begin{equation}\label{IU}
 A_q^{\mathcal{I}}=\bigcup_{p\in \mathcal{I}_q}B\left(\frac{p+\theta(q)}{q},\frac{\psi(q)}{q}\right) \pmod{1}.
\end{equation}

We divide the proof of Theorem \ref{thm1} into the following two propositions.

\begin{proposition}\label{p1}
    If $\sum_{q\in\mathbb{N}}\psi(q)\prod_{p\mid q,p\nmid B_q}(1-p^{-1})<\infty$, then~$\dim_{{\rm F}}W^{\ast}(\psi,\theta)\le 2s(\psi).$
\end{proposition}

\begin{proposition}\label{p2}
    $\dim_{{\rm F}}W^{\ast}(\psi,\theta)\ge \min\{2s(\psi),1\}.$
\end{proposition}

\subsection{Proof of Proposition \ref{p1}}
We proceed by contradiction. To this end, we assume that
\[
\dim_{\mathrm{F}} W^{\ast}(\psi,\theta) = 2s > 2s(\psi),
\]
which implies the existence of a Borel probability measure \(\nu\) supported on \(W^{\ast}(\psi,\theta)\) such that its Fourier transform satisfies
\[
|\widehat{\nu}(\xi)| \ll |\xi|^{-s} \quad \text{for } |\xi| \ge 1.
\]
Since $s > s(\psi)$, for any $0 < \varepsilon < s - s(\psi)$ we have
\[
\sum_{q \in \mathbb{N}} \left( \frac{\psi(q)}{q} \right)^{s - \varepsilon} < \infty.
\]

We will obtain a contradiction by showing that $\nu(W^{\ast}(\psi,\theta)) = 0$. This is achieved by exploiting the limsup set structure of $W^{\ast}(\psi,\theta)$ and applying the first Borel--Cantelli lemma. To begin, we estimate the $\nu$-measure of $A_q^{\mathcal{I}}$. For $p \in I_q$, define
\[
A_{p,q} = B\!\left( \frac{p + \theta(q)}{q}, \frac{\psi(q)}{q} \right) \pmod{1}.
\]
Since \(\psi(q) < 1/2\), the sets \(\{A_{p,q}\colon p \in \mathcal{I}_q\}\) are pairwise disjoint. Consequently, by (\ref{IU}),
\[
\nu(A_q^{\mathcal{I}}) = \sum_{p \in \mathcal{I}_q} \nu(A_{p,q}).
\]

Let $\mathbf{I}_{A_{p,q}}$ denote the indicator function of $A_{p,q}$,  extended to be periodic with respect to \(\mathbb{Z}\).  As usual, we write $e(x) = \exp(2\pi i x)$. The Fourier decay of $\nu$ ensures the convergence of the series $\sum_{k \in \mathbb{Z}} \widehat{\mathbf{I}}_{A_{p,q}}(k) \overline{\widehat{\nu}(k)}$, where
\[
\widehat{\mathbf{I}}_{A_{p,q}}(k) = \int_{B\left( \frac{p+\theta(q)}{q}, \frac{\psi(q)}{q} \right)} e(-k x) \, \mathrm{d}x, \] and \[ \widehat{\mathbf{I}}_{A_{p,q}}(0)=\frac{2\psi(q)}{q},\quad
|\widehat{\mathbf{I}}_{A_{p,q}}(k)| \ll \frac{1}{|k|}~~\big(k\in \mathbb{Z}\setminus \{0\}\big).
\]
Applying Parseval's identity, we obtain
\[
\nu(A_{p,q}) = \sum_{k \in \mathbb{Z}} \widehat{\mathbf{I}}_{A_{p,q}}(k) \overline{\widehat{\nu}(k)}.
\]
Therefore,
\begin{align*}
\nu(A_q) &= \sum_{p \in \mathcal{I}_q} \sum_{k \in \mathbb{Z}} \widehat{\mathbf{I}}_{A_{p,q}}(k) \overline{\widehat{\nu}(k)} \\
&= \sum_{p \in \mathcal{I}_q} \frac{2\psi(q)}{q} + \sum_{p \in \mathcal{I}_q} \sum_{k \in \mathbb{Z} \setminus \{0\}} \widehat{\mathbf{I}}_{A_{p,q}}(k) \overline{\widehat{\nu}(k)} \\
&= 2\psi(q) \prod_{ {p \mid q,~p \nmid B_q}} \bigl(1 - p^{-1}\bigr) + \sum_{p \in \mathcal{I}_q} \sum_{k \in \mathbb{Z} \setminus \{0\}} \widehat{\mathbf{I}}_{A_{p,q}}(k) \overline{\widehat{\nu}(k)}.
\end{align*}
The last equality follows from \cite[Proposition 1]{BHV24}: the cardinality of \( \mathcal{I}_q\) is
\begin{equation}\label{cardIq}
  \# \mathcal{I}_q = \phi(q) \prod_{p \mid (q, B_q)} \left(1 + \frac{1}{p-1}\right) = q \prod_{ {p \mid q,~p \nmid B_q}} \bigl(1 - p^{-1}\bigr),
\end{equation}
where we have used the standard identity $\displaystyle \frac{q}{\phi(q)} = \prod_{p \mid q} \left(1 + \frac{1}{p-1}\right)$.

To facilitate subsequent estimates, we prove the following auxiliary lemma giving a uniform upper bound for the general Ramanujan sum.
\begin{lem}\label{grs}
Let $A, B$ be coprime integers with $B>0.$ Then for $q,k\in\mathbb{N},$ we have
\[\left|\sum_{\substack{0\le p\le q-1 \\ \gcd(Bp+A,q)=1}}e(-k\cdot \frac{p}{q})\right|\ll \gcd(q,k)\cdot \log q,\]
where the implied constant in Vinogradov’s notation is independent of $A,B,q,k.$
\end{lem}

\begin{proof}
Recall the well-known property of the M\"{o}bius function $\mu\colon$

\[
\sum_{l \mid d} \mu(l) =
\begin{cases}
0 & \text{if } d > 1, \\
1 & \text{if } d = 1.
\end{cases}
\]
Then, we have
\begin{align*}
\sum_{\substack{0\le p\le q-1 \\ \gcd(Bp+A,q)=1}}e(-k\cdot \frac{p}{q})&=\sum_{p=0}^{q-1}e(-k\cdot \frac{p}{q})\cdot\sum_{l\mid \gcd(Bp+A,q)}\mu(l).
\end{align*}
If \(\gcd(l, B) \neq 1\), then since \(\gcd(A, B) = 1\), we have \(l \nmid Bp + A\) for any \(p \in \mathbb{Z}\).
If, on the other hand, \(\gcd(l, B) = 1\), we can find an inverse \(B^{(l)} \in \mathbb{N}\) of \(B\) modulo \(l\), i.e., \(B^{(l)} B \equiv 1 \pmod{l}\). Consequently,
\[
Bp + A \equiv 0 \pmod{l} \quad \iff \quad p \equiv -B^{(l)} A \pmod{l}.
\]
Thus we obtain
\begin{align*}
  \sum_{\substack{0\le p\le q-1 \\ \gcd(Bp+A,q)=1}}e(-k\cdot \frac{p}{q})
  &=\sum_{l\mid q }\mu(l)\cdot\sum_{\substack{0\le p\le q-1 \\ l\mid Bp+A}} e(-k\cdot \frac{p}{q})\\
  &=\sum_{\substack{l\mid q \\ \gcd(l,B)=1}}\mu(l)\cdot\sum_{\substack{0\le p\le q-1 \\ p \equiv -B^{(l)} A\pmod{l}}} e(-k\cdot \frac{p}{q}).
  \end{align*}
For the inner sum, we have
\[
\sum_{\substack{0 \leq p \leq q-1 \\ p \equiv -B^{(l)}A \pmod{l}}} e\!\left(-k \cdot \frac{p}{q}\right)
= \sum_{t=0}^{q/l - 1} e\!\left(-k \cdot \frac{p_0 + lt}{q}\right)
= e\!\left(-k \cdot \frac{p_0}{q}\right) \sum_{t=0}^{q/l - 1} e\!\left(-k \cdot \frac{t}{q/l}\right),
\]
where \(p_0\) is the unique residue modulo \(l\) satisfying \(p_0 \equiv -B^{(l)}A \pmod{l}\) with \(0 \le p_0 < l\).
The inner geometric sum vanishes unless \((q/l) \mid k\), in which case it equals \(q/l\). That is,
\[
\sum_{\substack{0 \leq p \leq q-1 \\ p \equiv -B^{(l)}A \pmod{l}}} e\!\left(-k \cdot \frac{p}{q}\right)
=
\begin{cases}
e\!\left(-k \cdot \frac{p_0}{q}\right) \cdot \dfrac{q}{l} & \text{if } \dfrac{q}{l} \mid k,\\[6pt]

0 & \text{otherwise}.
\end{cases}
\]
Inserting this into the outer sum over \(l \mid q\) and \(\gcd(l, B) = 1\) yields
\[
\sum_{\substack{l\mid q \\ \gcd(l,B)=1}} \mu(l) \sum_{\substack{0 \leq p \leq q-1 \\ p \equiv -B^{(l)}A \pmod{l}}} e\!\left(-k \cdot \frac{p}{q}\right)
= \sum_{\substack{l \mid q,\; \frac{q}{l} \mid k \\ \gcd(l,B)=1}} \mu(l) \, e\!\left(-k \cdot \frac{p_0}{q}\right) \frac{q}{l}.
\]
Taking absolute values, noting that \(|\mu(l)|\le 1\) and \(|e(x)| = 1\), and then removing the restriction condition \(\gcd(l,B)=1\) on the summation index \(l\), we obtain
\[
\left|\sum_{\substack{0\le p\le q-1 \\ \gcd(Bp+A,q)=1}} e\!\left(-k\cdot \frac{p}{q}\right)\right|
\le \sum_{l\mid q,\; \frac{q}{l}\mid k} \frac{q}{l}.
\]
If \(l\mid q\) and \((q/l)\mid k\), then \((q/l)\mid \gcd(q,k)\), and hence \(q\mid l\gcd(k,q)\). Setting \(t = l\gcd(k,q)/q\), we deduce that
\[
\sum_{l\mid q,\; \frac{q}{l}\mid k} \frac{q}{l}
\le \sum_{t=1}^{q} \frac{\gcd(q,k)}{t}
\ll \gcd(q,k)\cdot \log q,
\]
as desired.
\end{proof}

\begin{lem}\label{v-meas}
    For $\varepsilon>0$ we have
    \[\nu(A_q)\ll 2\psi(q) \prod_{ {p\mid q,~ p\nmid B_q}} \bigl(1 - p^{-1}\bigr)+ \left(\frac{\psi(q)}{q}\right)^{s-\varepsilon}.\]
\end{lem}

\begin{proof}
Recall that
\[
\nu(A_q)=2\psi(q) \prod_{ {p\mid q,~p\nmid B_q}} \bigl(1 - p^{-1}\bigr)+\sum_{p\in\mathcal{I}_q}\sum_{k\in\mathbb{Z}\setminus\{0\}}\widehat{\mathbf{I}}_{A_{p,q}}(k)\overline{\widehat{\nu}(k)}.
\]

We now estimate the second term on the right-hand side of the equality. Using the explicit form of the Fourier coefficient, we have
\begin{align*}
\sum_{p\in\mathcal{I}_q}\sum_{k\in\mathbb{Z}\setminus\{0\}}\widehat{\mathbf{I}}_{A_{p,q}}(k)\overline{\widehat{\nu}(k)}
&= \sum_{p\in\mathcal{I}_q}\sum_{k\in\mathbb{Z}\setminus\{0\}} e\!\left(-k \cdot \frac{p+\theta(q)}{q}\right) \int_{B(0,\frac{\psi(q)}{q})} e(-kx)\,\mathrm{d}x \cdot \overline{\widehat{\nu}(k)} \\
&\ll \log q \sum_{k=1}^\infty \gcd(q,k) \cdot \frac{\bigl| \sin\bigl(2\pi k \cdot \frac{\psi(q)}{q}\bigr) \bigr|}{k^{s+1}} \\
&= \log q \sum_{t=0}^{\infty}\sum_{j=1}^{q} \gcd(q,j) \cdot  \frac{\bigl| \sin\bigl(2\pi (tq+j) \cdot \frac{\psi(q)}{q}\bigr) \bigr|}{(tq+j)^{s+1}},
\end{align*}
where the inequality follows from Lemma \ref{grs}. %Here, $t \ge 0$ is an integer and $j$ runs over residue classes modulo $q$.
\medskip

\underline{\textbf{Case 1: $t = 0$.}}~  Using the bound $\sin x \le x$ for $x \ge 0$, we get
\[
\log q \sum_{j=1}^{q} \gcd(q,j) \cdot \frac{\bigl| \sin\bigl(2\pi j \cdot \frac{\psi(q)}{q}\bigr) \bigr|}{j^{s+1}}
\ll \log q \sum_{j=1}^{q} \gcd(q,j) \cdot \frac{1}{j^{s}} \cdot \frac{\psi(q)}{q}.
\]
Set $d = \gcd(q,j)$ and write $j = kd$ with $\gcd(k, q/d) = 1$. Then the right-hand side of the above inequality is
\[
\begin{aligned}
&=\frac{\psi(q)\log q}{q} \sum_{d\mid q} \sum_{\substack{1\le k \le q/d \\ \gcd(k, q/d)=1}} d \cdot \frac{1}{(kd)^{s}} \\
&= \frac{\psi(q)\log q}{q} \sum_{d\mid q} d^{1-s} \sum_{\substack{1\le k \le q/d \\ \gcd(k, q/d)=1}} \frac{1}{k^{s}} \\
&\le \frac{\psi(q)\log q}{q} \sum_{d\mid q} d^{1-s} \cdot \left(\frac{q}{d}\right)^{1-s} \\
&= \frac{\psi(q)\log q}{q^{s}} \cdot \tau(q)\ll \left(\frac{\psi(q)}{q}\right)^{s-\varepsilon},
\end{aligned}
\]
where $\tau(q)$ denotes the number of positive divisors of $q$. The last inequality uses the facts that $\tau(q) = O(q^{\varepsilon/2})$ and that the factor $\log q$ can be absorbed into the $q^{\varepsilon/2}$ term.
\medskip

\underline{\textbf{Case 2: $t \ne 0$.}}~In this case, we have
\begin{align*}
&\log q \sum_{t=1}^{\infty}\sum_{j=1}^{q} \gcd(q,j) \cdot  \frac{\bigl| \sin\bigl(2\pi (tq+j) \cdot \frac{\psi(q)}{q}\bigr) \bigr|}{(tq+j)^{s+1}}\\
\ll& \log q \sum_{j=1}^{q} \gcd(q,j) \cdot\left(\sum_{1\le t\le \frac{1}{\psi(q)}-1}\frac{(t+1)\psi(q)}{(tq)^{s+1}}+\sum_{t>\frac{1}{\psi(q)}-1}\frac{1}{(tq)^{s+1}}\right)\\
\ll & \left(\frac{\psi(q)}{q}\right)^{s-\varepsilon},
\end{align*}
where the last inequality follows from \cite[Theorem 3.2]{B01} that for all $\varepsilon>0,$
\[\sum_{j=1}^{q} \gcd(q,j)=O(q^{1+\frac{\varepsilon}{2}}).\]

Having shown that both cases lead to the desired estimate, we conclude the proof.

\end{proof}

Finally, by Lemma~\ref{v-meas}, we have
\[
\sum_{q\in\mathbb{N}} \nu(A_q) \ll \sum_{q\in\mathbb{N}} \left( 2\psi(q) \prod_{ {p\mid q,~p\nmid B_q}} \bigl(1 - p^{-1}\bigr) + \left(\frac{\psi(q)}{q}\right)^{s-\varepsilon} \right) < \infty.
\]
Applying the first Borel–Cantelli lemma, we conclude that \(\nu\bigl(W^{\ast}(\psi,\theta)\bigr) = 0\), which gives the desired contradiction needed to prove Proposition~\ref{p1}.

%\footnote{Mertens' theorem$\colon$ $\prod\limits_{p\le q}(1-\frac{1}{p})=\frac{e^{-\gamma}}{\log q}(1+O(\frac{1}{\log q})),$ where $\gamma$ is the Euler's constant.}.

\subsection{Proof of Proposition \ref{p2}}

We construct a probability measure $\nu$ supported on $W^{\ast}(\psi,\theta)$ with Fourier decay
\[
\widehat{\nu}(\xi)=O\bigl((1+|\xi|)^{-(\min\{2s(\psi),1\}-\varepsilon)}\bigr) \qquad (\xi \in \mathbb{R}),
\]
where $\varepsilon>0$. The measure \(\nu\) that we construct will be the limit of a sequence of absolutely continuous measures, the densities of which depend on a sequence \((M_k)\) (which will be constructed recursively as in Lemma~\ref{stability}). Our construction is a variant of the Cai and Hambrook construction.

Before stating Lemma \ref{stability}, we need to introduce some functions.

Fix a large integer \(K\), say \(K > s(\psi) + 4\), and let \(f \colon \mathbb{R} \to \mathbb{R}\) be a nonnegative function with \(\operatorname{supp}(f) \subseteq (-1,1)\) and \(\int_{\mathbb{R}} f(x) \, \mathrm{d}x = 1\), and
\begin{equation*}
   |\widehat{f}(\xi)| \ll (1 + |\xi|)^{-K} \qquad (\xi \in \mathbb{R}).
\end{equation*}
Let $q \in \mathbb{N}$. For $x \in \mathbb{R}$, define
\[
\Phi_{q,\theta}(x) = \frac{1}{\# \mathcal{I}_q} \cdot \frac{q}{\psi(q)} \sum_{p \in \mathcal{I}_q} f\!\left( \frac{x - (p+\theta(q))/q}{\psi(q)/q} \right),
\]
where $\mathcal{I}_q$ is defined in (\ref{I_q}). Applying Poisson summation, we obtain
\[
\Phi_{q,\theta}^{\ast}(x) := \sum_{k \in \mathbb{Z}} \Phi_{q,\theta}(x + k) = \sum_{k \in \mathbb{Z}} \widehat{\Phi_{q,\theta}}(k) e(kx),
\]
where the Fourier coefficients are given by
\begin{align*}
\widehat{\Phi_{q,\theta}}(k)
&= \frac{1}{\# \mathcal{I}_q} \cdot \frac{q}{\psi(q)} \sum_{p \in \mathcal{I}_q} \int_{\mathbb{R}} f\!\left( \frac{x - (p+\theta(q))/q}{\psi(q)/q} \right) e(-kx) \, \mathrm{d}x \\
&= \frac{1}{\# \mathcal{I}_q} \sum_{p \in \mathcal{I}_q} e\!\left( k \cdot \frac{p+\theta(q)}{q} \right) \widehat{f}\!\left( \frac{\psi(q)}{q} k \right).
\end{align*}
We note that  $\Phi_{q,\theta}^{\ast}(x)$ is 1-periodic, and  if $\Phi_{q,\theta}^{\ast}(x) > 0$, then there exists $p \in \mathcal{I}_q$ such that
\begin{equation}\label{positive}
    \left\| x - \frac{p+\theta(q)}{q} \right\| < \frac{\psi(q)}{q}.
\end{equation}

Recall that
\[
s(\psi) = \inf \left\{ s \in [0,1] : \sum_{q=1}^{\infty} \left( \frac{\psi(q)}{q} \right)^{\!s} < \infty \right\}.
\]
Then for any $\eta < s(\psi)$, we have
\[
\sum_{q=1}^{\infty} \left( \frac{\psi(q)}{q} \right)^{\!\eta}
= \sum_{k=0}^{\infty} \sum_{q \in Q(2^k)} \left( \frac{\psi(q)}{q} \right)^{\!\eta} = \infty,
\]
where
\[
Q(2^k) = \left\{ q \in \mathbb{N} : \frac{1}{2^{k+1}} \le \left( \frac{\psi(q)}{q} \right)^{\!\eta} < \frac{1}{2^{k}} \right\}.
\]
This implies that there are infinitely many $k$ such that
\begin{equation}\label{con1}
\sum_{q \in Q(2^k)} \left( \frac{\psi(q)}{q} \right)^{\!\eta} \ge \frac{1}{k^{2}}.
\end{equation}
Collect all such $k$ and define
\[
\mathcal{M} = \left\{ 2^{k} : \sum_{q \in Q(2^k)} \left( \frac{\psi(q)}{q} \right)^{\!\eta} \ge \frac{1}{k^{2}} \right\}.
\]
Then for each $M \in \mathcal{M}$, by \eqref{con1} we obtain
\begin{equation*}
\# Q(M) \cdot \frac{1}{M} \;\ge\; \sum_{q \in Q(M)} \left( \frac{\psi(q)}{q} \right)^{\!\eta} \;\gg\; \frac{1}{(\log M)^{2}}.
\end{equation*}
Consequently,
\begin{equation}\label{Q(M)}
\# Q(M) \;\gg\; \frac{M}{(\log M)^{2}}.
\end{equation}

For each $M \in \mathcal{M}$, we define a density function $g_M$ by
\[
g_M(x) := \frac{1}{\# Q(M)} \sum_{q \in Q(M)} \Phi_{q,\theta}^{\ast}(x)
      = \frac{1}{\# Q(M)} \sum_{q \in Q(M)} \sum_{k \in \mathbb{Z}} \widehat{\Phi_{q,\theta}}(k) e(kx).\]
\begin{lem}\label{subset}
    If $g_M(x)>0,$ then there exists $q\in Q(M)$ such that
    \[\left\|x-\frac{p+\theta(q)}{q}\right\|<\frac{\psi(q)}{q} \text{ for some } p\in \mathcal{I}_q.\]
\end{lem}

\begin{proof} This follows directly from \eqref{positive}.
%The \(1\)-periodicity of $g_M(x)$ implies the assertion stated in the lemma.
\end{proof}

\begin{lem}\label{gM_Fourier}
For $\varepsilon>0$, there exists $M_0^{\ast}>0$ such that for all $M\in\mathcal{M}$ with $M\ge M_0^{\ast}$, the Fourier coefficients of $g_M$ satisfy the following estimates:
\begin{enumerate}
\item[(1)] $\widehat{g_M}(0)=1$.
\item[(2)] For $l\in \mathbb{Z}\setminus \{0\}$,
\[
 \widehat{g_M}(l)  \ll
\begin{cases}
\dfrac{(\log M)^{5}\,\tau(l)}{M} & \text{if } 1\le |l|\le 3 M^{1/\eta}, \\[10pt]
|l|^{-(\eta-\varepsilon)} & \text{if } |l| > M^{1/\eta}.
\end{cases}
\]
%And thus $\widehat{g_M}(l)  \ll M^{-1+}$

\end{enumerate}
\end{lem}

\begin{proof} A straightforward computation of the Fourier coefficient yields
\begin{align*}
\widehat{g_M}(l)
&= \int_{0}^{1} \frac{1}{\#Q(M)} \sum_{q\in Q(M)} \sum_{k\in\mathbb{Z}} \widehat{\Phi_{q,\theta}}(k) e(kx) \, e(-lx) \, \mathrm{d}x \\
&= \frac{1}{\#Q(M)} \sum_{q\in Q(M)} \widehat{\Phi_{q,\theta}}(l) \\
&= \frac{1}{\#Q(M)} \sum_{q\in Q(M)} \frac{1}{\#\mathcal{I}_q} \sum_{p\in\mathcal{I}_q} e\!\left(l\cdot\frac{p+\theta(q)}{q}\right) \widehat{f}\!\left(\frac{\psi(q)}{q}l\right).
\end{align*}

\noindent(1) When \(l = 0\), since \(\widehat{f}(0) = 1\), we have \(\widehat{g_M}(0) = 1\).

\smallskip

\noindent(2) When \(l \neq 0\), We treat the two cases in the lemma separately.

\medskip

\noindent{}~~\underline{\textbf{Case 1: $1\le |l|\le 3M^{1/\eta}$.}} We bound the Fourier coefficients in the following way.
\begin{align*}
\bigl|\widehat{g_M}(l)\bigr|
&\le \frac{(\log M)^2}{M} \sum_{q\in Q(M)} \frac{1}{q\prod_{p\le q}(1-p^{-1})} \cdot \Bigl|\sum_{p\in\mathcal{I}_q} e\!\bigl(l\cdot \tfrac{p}{q}\bigr)\Bigr| \\
&\ll \frac{(\log M)^2}{M} \sum_{q\in Q(M)} \frac{(\log q)^2 \gcd(q,l)}{q} \\
&\ll \frac{(\log M)^4}{M} \sum_{q=1}^{(2M)\!^{\frac{1}{\eta}}} \frac{\gcd(q,l)}{q} \\
&\le \frac{(\log M)^4}{M} \sum_{d\mid l} \sum_{\substack{1\le k\le (2M)\!^{\frac{1}{\eta}}/d \\ (k,l/d)=1}} \frac{d}{dk} \\
&\ll \frac{(\log M)^5}{M} \, \tau(l),
\end{align*}
where the first inequality is a consequence of the bound \(|\widehat{f}(\xi)| \le 1\) (\(\xi \in \mathbb{R}\)) and the counting estimates \eqref{cardIq} and \eqref{Q(M)};
the second inequality follows from Lemma \ref{grs} and Mertens' third theorem\footnote{Mertens' third theorem states that $\lim_{q\to\infty} \log q \prod_{p\le q}(1-p^{-1}) = e^{-\gamma}$, where $\gamma$ is the Euler-Mascheroni constant.}; the third inequality holds because
\[\frac{1}{2M}\le \left(\frac{\psi(q)}{q}\right)^{\eta}< \left(\frac{1}{q}\right)^{\eta} \qquad \implies \qquad \log q \ll_{\eta} \log M.\]

\medskip

\noindent{}~~\underline{\textbf{Case 2: $|l| > M^{1/\eta}$.}}
Since $|\widehat{f}(\xi)| \ll (1+|\xi|)^{-K}$ with $K>\eta$, we have
\[
\bigl|\widehat{g_M}(l)\bigr|
\le \frac{(\log M)^2}{M} \sum_{q\in Q(M)} \frac{1}{q\prod_{p\le q}(1-p^{-1})} \cdot \Bigl|\sum_{p\in\mathcal{I}_q} e\!\bigl(l\cdot \tfrac{p}{q}\bigr)\Bigr| \cdot \frac{1}{\bigl(1+\frac{\psi(q)}{q}|l|\bigr)^{\eta}}.
\]
Since $({\psi(q)}/{q})^{-\eta}\asymp M$, we deduce that  $(1+ |l|{\psi(q)}/{q})^{-\eta}\ll M|l|^{-\eta}$. Proceeding as in Case~1 and noting that \(|l| > M^{1/\eta}\), we have for sufficiently large \(M\),
\begin{align*}
\bigl|\widehat{g_M}(l)\bigr|
\ll \frac{(\log M)^5}{M} \, \tau(l) \cdot M \, |l|^{-\eta}
\le |l|^{-(\eta-\varepsilon)}.
\end{align*}
This completes the proof.

\end{proof}

\begin{lem}\label{stability}
For every $h \in C_c^{K}(\mathbb{R})$ and $\delta > 0$, there exists   $M_0 = M_0(h, \delta)\in\mathbb{N}$ such that for all
$M \ge \max\{M_0,M_{0}^{\ast}\}$ (where $M_{0}^{\ast}$ is as in Lemma \ref{gM_Fourier}) and   $\xi \in \mathbb{R}$, we have the following estimate for the Fourier transforms:
\[
\bigl| \widehat{g_M h}(\xi) - \widehat{h}(\xi) \bigr| \ll \delta \, \bigl( 1 + |\xi| \bigr)^{-(\eta -\varepsilon)}.
\]
\end{lem}

\begin{proof}
 By  Lemma \ref{gM_Fourier}, we obtain
\begin{align*}
\bigl| \widehat{g_M h}(\xi) - \widehat{h}(\xi) \bigr|
 \le& \sum_{l\in\mathbb{Z}\setminus\{0\}} \bigl| \widehat{g_M}(l)  \bigr| \bigr|\widehat{h}(\xi-l)\bigr|\\
\le&  \sum_{\substack{l\in\mathbb{Z}\setminus\{0\}\\ |\xi-l|\ge \frac{1}{2}|\xi|}} \bigl|\widehat{g_M}(l)  \bigr|\cdot \frac{1}{(1+|\xi-l|)^{K}}+\sum_{\substack{l\in\mathbb{Z}\setminus\{0\}\\ |\xi-l|< \frac{1}{2}|\xi|}} \bigl|\widehat{g_M}(l)  \bigr|\cdot \frac{1}{(1+|\xi-l|)^{K}}\\
=:& ~S_1+S_2.
 \end{align*}
We estimate \(S_1\) and \(S_2\) separately.

\medskip

\noindent\textbf{Estimate of $S_1$.} By Lemma \ref{gM_Fourier}(2), we see that \(|\widehat{g_M}(l)| \ll \delta\) holds uniformly for \(l \in \mathbb{Z} \setminus \{0\}\) when \(M\) is large. Consequently,
\begin{align*}
S_1 &\ll \delta \sum_{\substack{l \in \mathbb{Z}\setminus\{0\} \\ |\xi - l| \ge \frac{1}{2}|\xi|}} \frac{1}{(1 + |\xi - l|)^{\eta}} \cdot \frac{1}{(1 + |\xi - l|)^{K - \eta}} \\
&\ll \frac{\delta}{(1 + |\xi|)^{\eta}} \sum_{\substack{l \in \mathbb{Z}\setminus\{0\} \\ |\xi - l| \ge \frac{1}{2}|\xi|}} \frac{1}{(1 + |\xi - l|)^{K - \eta}}
\ll \frac{\delta}{(1 + |\xi|)^{\eta}},
\end{align*}
where the last inequality follows from \(K - \eta > 4\) and hence the convergence of the series.
\medskip

\noindent\textbf{Estimate of $S_2$.} Note that in this regime we have $|l| \asymp |\xi|$. We distinguish two cases.

\medskip

\noindent\textbf{Case 1: $|\xi| \le 2M^{1/\eta}$.}
Applying Lemma \ref{gM_Fourier}(2) for $1 \le |l| \le 3M^{1/\eta}$, we obtain
\begin{align*}
S_2 &\ll \frac{(\log M)^{5} \tau(l)}{M} \sum_{\substack{l \in \mathbb{Z}\setminus\{0\} \\ |\xi - l| < \frac{1}{2}|\xi|}} \frac{1}{(1 + |\xi - l|)^{K - \eta}} \\
&\ll M^{-1+\varepsilon/(2\eta)} = M^{-\varepsilon/(2\eta)} M^{-1+\varepsilon/\eta} \ll \delta \, (1 + |\xi|)^{-(\eta - \varepsilon)},
\end{align*}
provided $M$ is large.

\medskip

\noindent\textbf{Case 2: $|\xi| > 2M^{1/\eta}$.}
In this case $|l| > M^{1/\eta}$, and Lemma \ref{gM_Fourier}(2) yields
\[
S_2 \ll |l|^{-(\eta - \varepsilon)} \sum_{\substack{l \in \mathbb{Z}\setminus\{0\} \\ |\xi - l| < \frac{1}{2}|\xi|}} \frac{1}{(1 + |\xi - l|)^{K - \eta}}
\le \delta \, (1 + |\xi|)^{-(\eta - \varepsilon)}
\]
for all large $M$.

Thus the proof of the lemma is complete.
\end{proof}

\noindent\textbf{Construction of the measure $\nu$.}
To prove Proposition \ref{p2}, we follow a construction analogous to that in \cite{B98} to obtain a measure \(\nu\) with Fourier decay
\[
\widehat{\nu}(\xi) = O\bigl(|\xi|^{-(\eta - \varepsilon)}\bigr).
\]
Roughly speaking, we construct a sequence of measures approximating \(\nu\) by repeatedly multiplying the density functions by \(g_{M_k}\), where the sequence \((M_k)_{k \in \mathbb{N}}\) is chosen recursively according to the stability lemma above.

\medskip
Fix a function \(h_0 \in C_c^{K}(\mathbb{R})\) such that \(\int h_0(x) \, \mathrm{d}x = 1\), \(\operatorname{supp}(h_0) = [0,1]\), and \(h_0(x) > 0\) for \(x \in [0,1]\).
Applying Lemma \ref{stability} iteratively, we choose
\begin{align*}
M_1 &= M_1(h_0, 2^{-1}),\\
M_2 &= M_2(h_0 g_{M_1}, 2^{-2}),\\
\vdots\\
M_k &= M_k(h_0 g_{M_1} \cdots g_{M_{k-1}}, 2^{-k}) \qquad (k \in \mathbb{N}),
\end{align*}
with \(\{M_k : k \ge 1\} \subset \mathcal{M}\) and \(2M_k < M_{k+1}\) for \(k \ge 1\). We then define the measures
\[
\mathrm{d}\nu_0 = h_0 \, \mathrm{d}x, \qquad
\mathrm{d}\nu_k = h_0 \, g_{M_1} \cdots g_{M_k} \, \mathrm{d}x \quad (k \in \mathbb{N}),
\]
which satisfy
\[
\bigl| \widehat{\nu}_k(\xi) - \widehat{\nu}_{k-1}(\xi) \bigr| \ll 2^{-k} \, |\xi|^{-(\eta - \varepsilon)}.
\]
This estimate implies that the sequence \(\{\widehat{\nu}_k\}_{k=0}^{\infty}\) is a  Cauchy sequence  in the supremum norm.
L\'evy's continuity theorem \cite[Theorem 3.3.17]{D19} then guarantees that \(\{\widehat{\nu}_k\}_{k=0}^{\infty}\) converges weakly to a finite non-zero Borel measure \(\nu\) satisfying
\[
\widehat{\nu}(\xi) = O\bigl(|\xi|^{-(\eta - \varepsilon)}\bigr).
\]

Furthermore, by Lemma \ref{subset}, we have
\[
\operatorname{supp}(\nu) \subseteq \bigcap_{k=1}^{\infty} \operatorname{supp}(g_{M_k}) \subseteq W^{\ast}(\psi,\theta).
\]

Normalizing \(\nu\) gives a probability measure. Then, by the definition of the Fourier dimension and the arbitrariness of \(\varepsilon\), the theorem follows.

\subsection*{Acknowledgements}
This work was supported by National Key R$\&$D Program of China (No. 2024YFA1013700) and NSFC (No. 12331005).

\subsection*{Declarations}

{\bf Conflict of interest} On behalf of all authors, the corresponding author states that there is no conflict of interest.

%\author{Bo Tan}
%{\footnotesize

%School  of  Mathematics  and  Statistics

%Huazhong  University  of Science  and  Technology, 430074 Wuhan, PR China

%Email: \texttt{tanbo@hust.edu.cn}}
%\vspace{5mm}

%\author{Qing-Long Zhou}
%{\footnotesize

%School  of  Mathematics  and  Statistics

% Wuhan University of Technology, 430070 Wuhan, PR China

%Email: \texttt{zhouql@whut.edu.cn}}

\end{document}